\documentclass[12pt]{amsart}
\usepackage[utf8]{inputenc}

\usepackage{amssymb, amsmath, amsthm}
\usepackage[margin=1in]{geometry}
\usepackage{tikz, tikz-cd}
\usepackage{graphicx}
\usepackage{enumerate}
\usepackage{comment}
\usepackage{bbold}
\usepackage{subcaption}
\usepackage{pinlabel}
\usepackage{hyperref} 
\usepackage{thmtools}
\usepackage{thm-restate}
\hypersetup{
    colorlinks = false, 
    linktoc = all,     
    linkcolor = darkgray,  
}
\usepackage[all,cmtip]{xy}
\usepackage{cleveref}

\usepackage[bordercolor=gray!20,backgroundcolor=blue!10,linecolor=blue!50,textsize=footnotesize,textwidth=0.8in]{todonotes}
\usepackage{pinlabel}

\theoremstyle{definition}
\newtheorem{theorem}{{Theorem}}[section]
\newtheorem{lemma}[theorem]{{Lemma}}
\newtheorem{proposition}[theorem]{{Proposition}}
\newtheorem{definition}[theorem]{{Definition}}
\newtheorem{corollary}[theorem]{{Corollary}}
\newtheorem{example}[theorem]{{Example}}
\newtheorem{conjecture}[theorem]{{Conjecture}}
\newtheorem{remark}[theorem]{{Remark}}

\newtheorem{question}[theorem]{Question}

\newcommand{\R}{\mathbb{R}}

\newcommand{\F}{\mathbb{F}}

\newcommand{\CFK}{\textit{CFK}} 

\newcommand{\vertical}{\text{vert}}

\newcommand{\s}{\mathfrak{s}}

\newcommand{\cc}{\mathbf{c}}
\newcommand{\yy}{\mathbf{y}}

\newcommand{\xx}{\mathbf{x}}

\title{On the Upsilon invariant of fibered knots and right-veering open books}

\author[D. He]{Dongtai He}
\address {Department of Mathematics, Boston College, Chestnut Hill, MA 02467}
\email{hedo@bc.edu}

\author[D. Hubbard]{Diana Hubbard}
\address {Department of Mathematics, Brooklyn College, Brooklyn, NY 11210}
\email{diana.hubbard@brooklyn.cuny.edu}

\author[L. Truong]{Linh Truong}
\address {School of Mathematics, Institute for Advanced Study, Princeton, NJ 08540}
\email{ltruong@math.ias.edu}

\begin{document}
\maketitle
\begin{abstract}
We give a sufficient condition using the Ozsv\'ath-Stipsicz-Szab\'o concordance invariant Upsilon for the monodromy of the open book decomposition of a fibered knot to be right-veering. As an application, we generalize a result of Baker on ribbon concordances between fibered knots. Following Baker, we conclude that either fibered knots $K$ in $S^{3}$ satisfying that $\Upsilon'(t) = -g(K)$ for some $t \in [0,1)$ are unique in their smooth concordance classes or there exists a counterexample to the Slice-Ribbon Conjecture.
\end{abstract}

\section{Introduction}

Under the Giroux correspondence, open book decompositions of a three-manifold $Y$, up to suitable equivalence, are in bijection with contact structures on $Y$.  
Associated to a fibered knot $K \subset Y$ and its fiber surface $\Sigma$ is an open book decomposition $(\Sigma, \phi_K)$ of $Y$, where $\phi_K: \Sigma \to \Sigma$ is the monodromy of $K$; the fibered knot is often referred to as the binding of the open book. We call $\xi_K$ the induced contact structure under the Giroux correspondence.  A contact structure $\xi_K$ can restrict properties of the monodromy $\phi_K$, as in the following theorem of Honda, Kazez, and Mati\'c.

\begin{theorem}[{\cite{HKM07}}] \label{tight}
A contact structure $\xi$ is tight if and only if the monodromy $\phi$ of every open book $(\Sigma, \phi)$ compatible with $\xi$ is right-veering.
\end{theorem}
Roughly, $\phi$ is right-veering if it sends every properly embedded arc in $\Sigma$ to the right of itself (see Section \ref{sec:fiberedknots} for a precise definition). Honda, Katez and Mati\'{c} also prove that any contact structure admits a right-veering open book via positive stabilization \cite{HKM07}. The question of how monodromies of open books dictate contact geometric information about the three-manifold is not thoroughly understood. 

It is therefore interesting to investigate connections between invariants of fibered knots and properties of their monodromies. The knot Floer complex \cite{OS-tau, Ras03} and its many associated knot invariants (see  \cite{manolescu2016introduction, Hom-survey} for surveys) have proven to be very powerful tools for studying knots. For instance, knot Floer homology detects the genus of knots \cite{OS04b} and detects whether a knot is fibered \cite{ni2007knot}; the knot Floer complex has also given rise to several concordance invariants, starting with the Ozsv\'ath-Szab\'o $\tau$ invariant \cite{OS-tau}. Hedden in \cite{Hedden} produces a relationship between the $\tau$ invariant of a fibered knot in $S^{3}$ and the corresponding contact structure:
\begin{theorem}[{\cite{Hedden}}]
The contact structure $\xi_K$ is {tight} if and only if $\tau(K) = \text{genus}(K)$. 
\end{theorem}

Ozsv\'ath-Stipsicz-Szab\'o \cite{OSS} constructed a one-parameter family of concordance invariant $\Upsilon_K(t)$ that is stronger than the $\tau(K)$ invariant. It is natural to ask whether $\Upsilon_K(t)$ contains more information about the monodromy of a fibered knot than $\tau(K)$. In the main theorem of this paper, we give an affirmative answer by providing a sufficient condition for the monodromy to be {right-veering}.

\begin{theorem}\label{thm:main}
Suppose $K$ is a null-homologous fibered knot in a rational homology sphere $Y$. If $\Upsilon'_{K,\mathfrak{s}}(t)=-g$ for some $t\in [0,1)$, where $g$ is the genus of the fibered surface $\Sigma$, then $\phi:\Sigma\rightarrow\Sigma$ is right-veering. 
\end{theorem}

Here by $\Upsilon(K)$ we mean the generalization of the invariant from \cite{OSS} to null-homologous knots in rational homology spheres introduced by Alfieri-Celoria-Stipsicz \cite{AlfieriCeloriaStipsicz}. In the case of $Y=S^{3}$, this is the same $\Upsilon$-invariant defined in \cite{OSS}.

In Example \ref{ex:cable} we provide an infinite family of fibered knots in $S^{3}$ which have right-veering monodromy by Theorem \ref{thm:main} but whose corresponding contact structures are not tight, showing that $\Upsilon$ detects finer information about the monodromies of fibered knots than $\tau$. In Example \ref{ex:converse} we show that the converse of Theorem \ref{thm:main} does not hold even in $S^{3}$.

In Section \ref{sec:applications} we discuss some consequences of Theorem \ref{thm:main}. One corollary is that under certain circumstances we can obstruct concordance between two fibered knots in $S^{3}$:

\begin{proposition}\label{prop:concordanceobst} Let $K$ in $S^{3}$ be a fibered knot of genus $g$ such that $\Upsilon'_{K}(t_{0}) = -g(K)$ for some $t_{0} \in [0,1)$. Then $K$ cannot be concordant to any fibered knot in $S^{3}$ of the same genus whose monodromy is not right-veering. 
\end{proposition} 

In general one fibered knot having right-veering monodromy and the other non-right-veering is not enough to obstruct concordance, even if the two knots have the same genus. We find it interesting that under certain additional conditions these concordances cannot occur. A natural question to ask, and one that we think is deserving of more exploration, is whether there are other circumstances when monodromies of fibered knots can obstruct concordance.

We conclude the paper with an application to the unresolved Slice-Ribbon Conjecture that is inspired by the work of Baker \cite{baker}. The Slice-Ribbon Conjecture famously posits that every slice knot bounds a ribbon disk. We prove:

\begin{theorem}\label{thm:ribbonequal}
Let $K_{0}$ and $K_{1}$ be fibered knots in $S^{3}$ such that $\Upsilon'_{K_{i}}(t_{i}) = -g(K_{i})$ for some $t_{i} \in [0,1]$, for each $i \in \{0, 1\}$. If $K_{0}\# -K_{1}$ is ribbon, then $K_{0}=K_{1}$.
\end{theorem}

We ask whether the corresponding statement would be true if the ribbon condition were replaced by $K_{0}$ and $K_{1}$ being concordant:

\begin{question}\label{que:sliceribbon} Suppose $K_{0}$ and $K_{1}$ are fibered knots in $S^{3}$ satisfying that for each $i \in \{0, 1\}$, $\Upsilon'_{K_{i}}(t_{i}) = -g(K_{i})$ for some $t_{i} \in [0,1)$. If $K_{0}$ and $K_{1}$ are concordant, is $K_{0} = K_{1}$? 
\end{question}

And we immediately observe that this line of reasoning opens up a possible avenue for disproving the Slice-Ribbon Conjecture:

\begin{corollary}\label{cor:sliceribbon} If the answer to Question \ref{que:sliceribbon} is negative, then the Slice-Ribbon Conjecture is false.
\end{corollary}

Finally, we note here that this work was first inspired by \cite{GLW16}. In the setting of oriented links in a thickened annulus, Grigsby, Licata and Wehrli in \cite{GLW16} define a family of annular Rasmussen invariants $\{d_{t}(L,o)\}_{t\in [0,2]}$ from the Khovanov-Lee complex. In particular, the authors study the case when $(L,o)$ is a braid closure $L = \hat{\beta}$ equipped with its braid-like orientation $o$. They find interesting connections between $d_{t}(\hat{\beta})$ and the positivity of braids and the right-veeringness property of the monodromy:

\begin{theorem} \label{bp}
\cite{GLW16} Let $\hat{\beta}$ be a braid closure with its natural orientation. If $\beta$ is quasipositive, then $d'_{t}(\hat{\beta})=b$ for all $t\in [0,1)$, where $b$ is the braid index of $\beta$.
\end{theorem}    

\begin{theorem} \label{br}
\cite{GLW16} If $d'_{t}(\hat{\beta})=b$ for some $t\in [0,1)$, then $\beta$ is right-veering. 
\end{theorem}
Recall from \cite{Birman-Hilden} that the preimage of the braid axis in the double branched cover of a braid closure is a fibered knot. Inspired by the Birman-Hilden correspondence between braid closures and fibered knots in the double branched cover, Theorem \ref{thm:main} uses the $\Upsilon$-invariant for fibered knots, analogous to the way that Theorem \ref{br} uses the $d_t$ invariant for braid closures, to prove a sufficient condition for right-veeringness. However, we note that the analogue of Theorem \ref{bp} does not hold, as the $\Upsilon-$invariant for a quasipositive fibered knot (e.g. the torus knot $T(3,7)$) does not necessarily have a single slope on $t\in [0,1)$.

\subsection{Acknowledgements}
The main theorem of this paper grew out of ideas from the first author's PhD thesis \cite{He-thesis}. We thank Eli Grigsby for suggesting and encouraging this collaboration and her interest in this project. We also thank Jen Hom and Olga Plamenevskaya for thoughtful comments. The second and third authors would like to thank the organizers of the 2019 \emph{Women in Symplectic and Contact Geometry and Topology} (WiSCon) workshop at ICERM, where some of this work took place, and are grateful to ICERM and to the AWM for supporting WiSCon via the AWM ADVANCE grant NSF-HRD 1500481. The second author was  supported in part by an AMS-Simons travel grant. The third author was partially supported by NSF grant DMS-1606451 and the Institute for Advanced Study.

\section{Background on the Knot Floer complex}\label{sec:HF}
In this section we give a brief overview of the construction of the Heegaard Floer complex of knots due to \cite{Ras03} and independently \cite{OS04}. We assume that the reader is familiar with the basics of the knot Floer package of invariants and aim primarily to establish notation. For a detailed expository description, see \cite{manolescu2016introduction}. Let $Y$ be a rational homology sphere, and let $K\subset Y$ a null-homologous knot. We can associate to the pair $(Y,K)$ a doubly pointed Heegaard diagram $(\Sigma,\alpha,\beta,w,z)$ consisting of the following data:

\begin{itemize}
\item A Heegaard surface of genus $g$, splitting $Y$ into two handlebodies $U_{0}$ and $U_{1}$
\item linearly independent curves $\alpha=\{\alpha_{1},...,\alpha_{g}\}$, $\beta = \{\beta_{1},...,\beta_{g}\}$ on $\Sigma$ and
\item base points $w,z\in\Sigma-\alpha_{1}-...-\alpha_{g}-\beta_{1}-...-\beta_{g}$.
\end{itemize}
If we connect $w$ and $z$ by a curve $a$ in $\Sigma-\alpha_{1}-...-\alpha_{g}$ and another curve in $\Sigma-\beta_{1}-...-\beta_{g}$, then the knot $K$ is obtained by pushing $a$ and $b$ into $U_{0}$ and $U_{1}$ respectively.

Consider the set of points in $\mathbb{T}_{\alpha} \cap \mathbb{T}_{\beta}$, that is, consider the usual $g$-tuples of intersection points between the $\alpha$- and $\beta$-curves, where each $\alpha$ and each $\beta$ curve is used exactly once. This set admits two gradings: the Maslov (homological) grading and Alexander grading. 
For any $\mathbf{x}, \mathbf{y} \in \mathbb{T}_{\alpha} \cap \mathbb{T}_{\beta}$ and $\phi \in \pi_{2}(\mathbf{x},\mathbf{y})$, the Alexander grading $A(\mathbf{x})$ satisfies that  $$A(\mathbf{x})-A(\mathbf{y})=n_{z}(\phi)-n_{w}(\phi),$$  and the Maslov grading $M(\mathbf{x})$ satisfies that $$M(\mathbf{x})-M(\mathbf{y}) = \mu(\phi) -2 n_{w}(\phi).$$

The \textit{(full) knot Floer complex} $CFK^{\infty}(Y,K)$ is freely generated by $\mathbb{T}_{\alpha} \cap \mathbb{T}_{\beta}$ over $\mathbb{F}_{2}[U,U^{-1}]$. Following \cite{OS04} and \cite{manolescu2016introduction}, we think of $CFK^{\infty}$ as being freely generated over $\mathbb{Z}$ by triples 
$$[\mathbf{x},i,j], \mathbf{x} \in \mathbb{T}_{\alpha} \cap \mathbb{T}_{\beta}, i, j \in \mathbb{Z}, \text{\, with \,} A(\mathbf{x}) = j-i.$$ The triple $[\mathbf{x},i,j]$ designates $U^{-i} \mathbf{x}$. The $U$-action decreases the Alexander grading by one, so the $j$-coordinate of $[\mathbf{x},i,j]$ describes the Alexander grading of $U^{-i} \mathbf{x}$ and the $i$-coordinate describes the negative of the $U$-power. So we have:
 $$U([\mathbf{x},i,j])=[\mathbf{x},i-1,j-1]$$ The Maslov grading is decreased by two by the $U$-action. 
The differential on $CFK^{\infty}$ is given by $$\partial^{\infty}[\mathbf{x},i,j]=\sum_{\mathbf{y}\in\mathbb{T}_{\alpha}\cap\mathbb{T}_{\beta}}\sum_{\{\phi\in\pi _{2}(\mathbf{x},\mathbf{y}|\mu (\phi)=1\}}\# (\widehat{M}(\phi))[\mathbf{y},i-n_{w}(\phi ),j-n_{z}(\phi)]$$
where $\# (\widehat{\mathcal{M}}(\phi))$ is counted modulo 2. The filtered chain homotopy type of $CFK^{\infty}$ is a knot invariant.

It is often useful to represent $CFK^{\infty}$ graphically. We do so by representing each $[\mathbf{x},i,j]$ as a dot in the plane at coordinate $(i,j)$ and drawing the differentials as arrows. (In such a picture the Maslov grading is suppressed.) The differential on the complex decreases the Maslov grading by one, and can preserve or decrease the Alexander grading. On the graphical representation of $CFK^{\infty}$, the arrows point non-strictly down and left.
 Finally, $CFK^{\infty}$ splits as a direct sum, that is, $$CFK^{\infty}(Y,K)=\bigoplus_{\mathfrak{s}\in spin^{c}(Y)}CFK^{\infty}(Y,K,\mathfrak{s}),$$
and the homology of $CFK^{\infty}(Y,K,\mathfrak{s})$ is $HF^{\infty}(Y,\mathfrak{s})\cong\mathbb{F}[U,U^{-1}]$ as a relatively graded $\mathbb{F}[U,U^{-1}]-$module. An absolute grading can be defined where the base element $\mathbf{1}\in\mathbb{F}[U,U^{-1}]$ has Maslov grading $d(Y,\mathfrak{s})$, the Heegaard Floer correction term defined in \cite{OS03}. 

We can consider $CFK^{\infty}$ as a $\mathbb{Z}\oplus\mathbb{Z}$ filtered complex, where the $i$-coordinate gives what we call the algebraic filtration, and the $j$-coordinate gives what we call the Alexander filtration. Further on in the paper, following \cite{Hom} we will consider only the parts of the differential that preserve the Alexander grading (that is, the horizontal arrows), and this will be denoted $\partial^{horz}$, and we will also consider only the parts of the differential represented by vertical arrows, $\partial^{vert}$. Note that many of the other knot Floer and Heegaard Floer constructions can be obtained from this complex by restricting which $(i,j)$ we consider. In particular, if we consider only triples of the form $[\mathbf{x},0,j]$ and restrict to the differentials that preserve $i$ and $j$, then we obtain the well-studied hat complex $\widehat{CFK}$.  If we consider only triples of the form $[\mathbf{x},0,j]$ and restrict to the differentials that preserve $i$ (but allow $j$ to change), then we obtain $\widehat{CF}$ with the Alexander filtration given by $j$.  Later on in the paper, we will refer to these filtration levels as $\mathcal{F}_{j}$, that is, $\mathcal{F}_{j}$ will represent the subspace of $\widehat{CF}$ spanned by generators with Alexander grading at most $j$.

\section{The Upsilon invariant for knots in rational homology spheres}


In this section we review the definition and properties of the concordance invariant Upsilon for null-homologous knots in rational homology spheres introduced by Alfieri-Celoria-Stipsicz \cite{AlfieriCeloriaStipsicz}. The definition closely follows Livingston's approach in \cite{Liv17} for defining the Upsilon invariant for knots in $S^3$.

Fix a null-homologous knot $K$ in a rational homology sphere $Y$ and fix a $Spin^{c}$ structure on $Y$.  For $t\in [0,2]$ and a generator $[\mathbf{x},i,j] \in CFK^{\infty}(Y,K,\mathfrak{s})=C$, we define the real-valued function $$f_{t}([\mathbf{x},i,j])=(1-\frac{t}{2})i+\frac{t}{2}j.$$ Furthermore, if $\mathbf{\theta} = [\mathbf{x}_{1},i_{1},j_{1}] +...+ [\mathbf{x}_{n},i_{n},j_{n}]$ is a chain in $C$, we also define a function $$F_{t}(\mathbf{\theta})= \max\{f_{t}([\mathbf{x}_{k},i_{k},j_{k})]\}.$$


\begin{proposition} 
The function $F_{t}$ defines a filtration $\mathcal{F}^{t}$ on $C$, where the filtered subcomplexes are given by $\mathcal{F}^{t}_{s}=F_{t}^{-1}(-\infty,s]$. Furthermore, $\mathcal{F}^{t}$ is discrete, i.e., for any $s_{1}\geq s_{2}$, $\mathcal{F}^{t}_{s_{1}}/\mathcal{F}^{t}_{s_{2}}$ is finite-dimensional.   
\end{proposition}

\noindent
\begin{proof} We have $\partial^{\infty}(\mathbf{\theta})=\Sigma\partial^{\infty}[\mathbf{x}_{k},i_{k},j_{k}]$, where $\partial^{\infty}$ non-strictly reduces both $i_{k}$ and $j_{k}$. Both $1-\frac{t}{2}$ and $\frac{t}{2}$ are nonnegative, so $F_{t}(\mathbf{\theta})\geq F_{t}(\partial^{\infty}(\mathbf{\theta}))$. For discreteness, notice first that for $t \in (0,2]$ the filtered subcomplex $\mathcal{F}^{t}_{s}$ is generated by triples $[\mathbf{x},i,j]$ in the $(i,j)$-plane that lie on or below the line of slope $1-\frac{2}{t}$ with $j$-intercept $\frac{2s}{t}$; when $t=0$ we are considering just those triples with $i \leq s$. From this graphical perspective we see that there must exist $k_{1}$ and $k_{2}$ such that $C(i\leq k_{1})\subset\mathcal{F}^{t}_{s_{2}} \subset\mathcal{F}^{t}_{s_{1}}\subset C(i\leq k_{2})$. Since the algebraic filtration is discrete, so is $\mathcal{F}^{t}$. \end{proof}


The Alfieri-Celoria-Stipsicz invariant $\Upsilon$ depends on the choice of a southwest region of the plane. We work with the southwest region corresponding to a half-plane
\[ U_t = \{ (x,y) \in \R^2 \mid \frac{t}{2} \cdot y + (1- \frac{t}{2}) \cdot x \leq 0\}.\]
Given the knot Floer complex $\CFK^\infty(K)$, we will denote by $C_t$ the $\F_2[U]$-submodule of $\CFK^\infty(K)$ spanned by the generators lying in $U_t$.
\begin{definition}[{\cite{AlfieriCeloriaStipsicz}}] \label{def:ACS}
 Let $t \in [0,2]$. Define
\[\nu_t(Y, K, \s) = \min \{ s \mid (C_t)_s \hookrightarrow \CFK^\infty(Y,K,\s) \text{ induces a surjective map on } H_{d(Y,\mathfrak{s})} \}\]
where $(C_t)_s = \{ (x,y) \mid (x-s, y-s) \in C_t\}$
and the Maslov grading gives the homological grading. Define $\Upsilon_{Y,K,\mathfrak{s}}(t)=-2\nu_{t}(Y,K,\mathfrak{s})$.
\end{definition}
The $\Upsilon_{Y, K, \s}$ invariant in Definition \ref{def:ACS} recovers the Livingston reformulation of $\Upsilon_K$ for knots in $S^3$ \cite{Liv17}. See \cite[Sections 3 and 4]{AlfieriCeloriaStipsicz}.

We will work with the following equivalent definition (see also \cite[Definition 4.2]{HomLevineLidman}).
\begin{definition}[{cf. Definition \ref{def:ACS}}]\label{nu} Let $t \in [0,2]$. Define
\[\nu_{t}(Y,K,\mathfrak{s})= \min \{F_{t}(\theta) \mid \theta \text{ is a cycle in }C \text{ with Maslov grading } d(Y,\mathfrak{s}), 0 \neq [\theta] \in H_*(C)\}.\]
Define
$\Upsilon_{Y,K,\mathfrak{s}}(t)=-2\nu_{t}(Y,K,\mathfrak{s})$.
\end{definition} 

When $Y$ is understood from the context, then we drop it from the notation. We say a generator $[\mathbf{x},i,j]$ realizes $\Upsilon_{K,\mathfrak{s}}(t)$ if $[\mathbf{x},i,j]$ is a summand of a cycle $\theta$ satisfying the condition in Definition \ref{nu} and $\nu_{t}(K,\mathfrak{s})=f_{t}([\mathbf{x},i,j])$.


An initial observation is that $\Upsilon_{K,\mathfrak{s}}(0)=0$. Indeed, $f_{0}([\mathbf{x},i,j])=i$ is the algebraic filtration. 


\begin{theorem}[cf. {\cite[Theorem 7.1]{Liv17}}] \label{func}
Given $t\in [0,2]$,
\begin{enumerate}
\item $\Upsilon_{K,\mathfrak{s}}(t)$ is a continuous piecewise linear function.
\item If $\Upsilon_{K,\mathfrak{s}}(t)$ is differentiable at $t$, and a generator $[\mathbf{x},i,j]$ realizes $\Upsilon_{K,\mathfrak{s}}(t)$, then $\Upsilon'_{K,\mathfrak{s}}(t)=i-j=-A(\mathbf{x})$.
\item $\Upsilon_{K,\mathfrak{s}}(t)$ is not differentiable at $t$ only if at least two generators $[\mathbf{x},i,j]$, $[\mathbf{x'},i',j']$ realize $\Upsilon_{K,\mathfrak{s}}(t)$. Moreover, $(i,j)$ and $(i',j')$ lie on the same line of slope $1-\frac{2}{t}$. Let $$\bigtriangleup\Upsilon'_{K,\s}(t) = \lim_{t\searrow t_{0}}\Upsilon'_{K,\s}(t)-\lim_{t\nearrow t_{0}}\Upsilon'_{K,\s}(t),$$ then $\bigtriangleup\Upsilon'_{K,\s}(t)=\frac{2}{t}(i'-i)$.
\end{enumerate}
\end{theorem}


\begin{proof} The proof is essentially the same as \cite{Liv17}. Since $\mathcal{F}^{t}$ is discrete, for all but finitely many $t$ there is exactly one generator $[\mathbf{x},i,j]$ realizing $\Upsilon_{K}(t)$. For nearby $t$, say $t_{1}$, $\Upsilon_{K}(t_{1})$ is realized by the same generator $[\mathbf{x},i,j]$ so that $\nu_{t_{1}}(K,\mathfrak{s})=(1-\frac{t_{1}}{2})i+\frac{t_{1}}{2}j$. Written differently, 
$$\Upsilon_{K,\mathfrak{s}}(t)=-2\nu_{t}(K,\mathfrak{s})=(i-j)t-2i.$$
Therefore,
\begin{enumerate}
\item $\Upsilon_{K,\s}(t)$ is a continuous piece-wise linear function.
\item $\Upsilon'_{K,\mathfrak{s}}(t)=i-j=-A(\xx)$ if $(i,j)$ is the only lattice point that realizes $\Upsilon_{K,\s}(t)$. 
\item On the other hand, $\Upsilon_{K,\mathfrak{s}}(t)$ is not differentiable only if two generators $[\mathbf{x},i,j]$, $[\mathbf{x'},i',j']$ realize $\Upsilon_{K,\mathfrak{s}}(t)$ and $i-j\neq i'-j'$. Furthermore, $$\Upsilon_{K,\s}(t)=(i-j)t-2i=(i'-j')t-2i'.$$
Written differently,
$$\frac{j-j'}{i-i'}=1-\frac{2}{t}.$$
Therefore, $(i,j)$ and $(i',j')$ lie on the same line of slope $1-\frac{2}{t}$ in the $ij-$plane.

$\bigtriangleup\Upsilon'_{K,\s}(t) = (j-j')-(i-i')$, so $\bigtriangleup\Upsilon'_{K,\s}(t)=\frac{2}{t}(i'-i)$.
\end{enumerate}
\end{proof}

\begin{remark}
We define $\Upsilon_{K,\s}'(0)=\lim\limits_{t\searrow 0}\Upsilon_{K,\s}'(t)$ and $\Upsilon_{K,\s}'(2)=\lim\limits_{t\nearrow 2}\Upsilon_{K,\s}'(t)$.
\end{remark}

\begin{corollary}
$|\Upsilon'_{K,\mathfrak{s}}(t)|\leq g(K)$. 
\end{corollary}

\begin{proof} If $[\xx,i,j]$ realizes $\Upsilon_{K,\s}(t)$, then $|\Upsilon'_{K,\mathfrak{s}}(t)|=A(\xx)\leq g(K)$.
\end{proof}


\begin{theorem}[Proposition 4.1 of \cite{AlfieriCeloriaStipsicz}] \label{basic}
The $\Upsilon-$invariant satisfies the following properties: 
\begin{enumerate}
\item $\Upsilon_{Y\# Y',K\# K',\mathfrak{s}\#\mathfrak{s'}}(t)=\Upsilon_{Y,K,\mathfrak{s}}(t)+\Upsilon_{Y',K',\mathfrak{s'}}(t)$. 
\item $\Upsilon_{Y,K,\mathfrak{s}}(t)=-\Upsilon_{-Y,-K,\overline{\mathfrak{s}}}(t)$
\end{enumerate}
\end{theorem}
\begin{proof}
For part (a), the complex $CFK^{\infty}(Y\# Y',K\# K',\mathfrak{s}\#\mathfrak{s'})$ is bifiltered chain homotopy equivalent to $CFK^{\infty}(Y,K,\mathfrak{s})\otimes CFK^{\infty}(Y',K',\mathfrak{s'})$. If $(C,\mathcal{F})$ and $(C',\mathcal{F}')$ are two filtered complexes, there is a natural filtration $\mathcal{F}\otimes\mathcal{F}'$ on $C\otimes C'$:
\[(C\otimes C')_{s}=\text{Image}(\oplus_{s=s_{1}+s_{2}}C_{s_{1}}\otimes C'_{s_{2}}\rightarrow C\otimes C').\]
It follows from Theorem 6.1 in \cite{Liv17} that $\nu_{t}$ is additive for each $t$. Hence \[\Upsilon_{Y\# Y',K\# K',\mathfrak{s}\#\mathfrak{s'}}(t)=\Upsilon_{Y,K,\mathfrak{s}}(t)+\Upsilon_{Y',K',\mathfrak{s'}}(t).\]

Part (b) (and also (a)) follows from Proposition 4.1 of \cite{AlfieriCeloriaStipsicz}.
\end{proof}

\begin{proposition}
$\Upsilon_{K,\mathfrak{s}}(t)=\Upsilon_{K,\mathfrak{s}}(2-t)$.
\end{proposition}
\begin{proof}
This follows immediately from switching the role of base points $w$ and $z$. 
\end{proof}

\section{Fibered Knots}\label{sec:fiberedknots}

We start by reviewing the definition of a right-veering surface diffeomorphism \cite{HKM07}. 

\subsection{Right-veering diffeomorphisms}

Let $\Sigma$ be a compact oriented surface with boundary $\partial\Sigma$, and let $\alpha$, $\beta:[0,1]\rightarrow\Sigma$ be properly embedded oriented arcs with $\alpha (0)=\beta (0)=x\in\partial\Sigma$. Isotope $\alpha$ and $\beta$ so that they intersect transversely with the fewest possible number of intersections. We say that $\beta$ is to the right of $\alpha$ if $(\dot{\beta}(0),\dot{\alpha}(0))$ define the orientation of $\Sigma$ at $x$.

\begin{definition}\label{def:rv}
Let $\phi:\Sigma\rightarrow\Sigma$ be a diffeomorphism which restricts to the identity map on the boundary $\partial\Sigma$. Let $\alpha$ be a properly embedded oriented arc starting at a base point $x\in\partial\Sigma$. Then we say $\phi$ is right-veering if for an arbitrary base point $x$ and arc $\alpha$, $\phi(\alpha)$ is always to the right of $\alpha$. 
\end{definition}

\subsection{Knot Floer homology of fibered knots} 

Let $K$ be the binding of an open book $(\Sigma,\phi)$ of $Y$ compatible with a contact structure $\xi$. A basis for $\Sigma$ is a collection $\{a_{1},...,a_{2g}\}$ of disjoint, properly embedded arcs in $\Sigma$ whose complement is a disk. Let $b_{i}$ be an isotopic copy of $a_{i}$ obtained by shifting the end points of $a_{i}$ in the direction of $K$ so that $b_{i}$ intersects $a_{i}$ at a single point $c_{i}$. Following \cite{HKM09}, we form a pointed Heegaard diagram 
\[(S,\mathbf{\alpha}=(\alpha_{1},...,\alpha_{2g}),\mathbf{\beta}=(\beta_{1},...,\beta_{2g}),w)\]
for $Y$ by doubling the open book:
\begin{itemize}
\item $S=\Sigma\cup -\Sigma$ is the union of two copies of $\Sigma$ glued along the binding $K$,
\item $\alpha_{i}=a_{i}\cup a_{i}$,
\item $\beta_{i}=b_{i}\cup \phi(b_{i})$,
\item the base point $w$ lies outside of the strip from the isotopies from $a_{i}$ to $b_{i}$
\end{itemize}
\noindent
as shown in Figure \ref{heegaard}. 

\begin{figure}[h]
\centering
\includegraphics[scale=0.5]{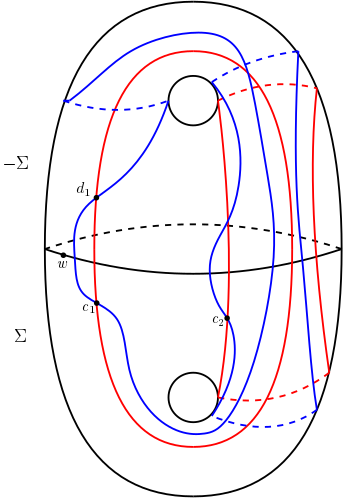}
\caption{A pointed Heegaard diagram compatible with an open book, following \cite{HKM09, BVV}. The curves $\alpha_{i}$ are shown in red and the $\beta_{i}$ are in blue.}\label{heegaard}
\end{figure}

Now we turn the Heegaard diagram into a doubly-pointed Heegaard diagram for $K\subset Y$. We perform finger moves on the $\mathbf{\beta}$ curves in a neighborhood of $\partial \Sigma$ in the direction of the orientation of $K$, and place the second base point $z$ inside the region of the isotopies: see Figure \ref{rvfigure}.
Later in the paper we will be most interested in the doubly-pointed Heegaard diagram $\mathcal{H} = (S, \mathbf{\beta}, \mathbf{\alpha}, z, w) $, obtained by reversing the roles of $\alpha$ and $\beta$ and the basepoints $w$ and $z$, for $-K \subset - Y$.

\begin{figure}[h]
\centering
\includegraphics[scale=0.6]{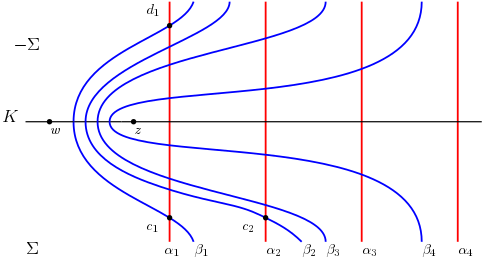}
\centering
\caption{A local picture of the doubly-pointed Heegaard diagram $\mathcal{H}$ for $K \subset Y$ after performing finger moves on the Heegaard diagram from Figure \ref{heegaard}.}\label{rvfigure}
\end{figure}

The following lemma by Baldwin and Vela-Vick \cite{BVV} characterizes the Alexander grading of generators.

\begin{lemma}[\cite{BVV}]
\label{lem:bvv}
The Alexander grading of a generator $\mathbf{x} \in \widehat{\CFK}(\mathcal{H})$ is the number of components in $-\Sigma\subset S$ minus $g$, where $g$ is the genus of $K$. 
\end{lemma}

\begin{proposition}[\cite{BVV}]
If $A(\mathbf{x})=-g$, then every component of $\xx$ lies in $\Sigma$.
\end{proposition}

We observe that each component of $\xx$ is therefore either a $c_{i}$ or is a new intersection point of the $\alpha$ and $\beta$ curves coming from the finger move. We show in Figure \ref{fig:newintersection} an example of how these new intersection points can arise.

\begin{figure}[h]
\centering
\includegraphics[scale=0.6]{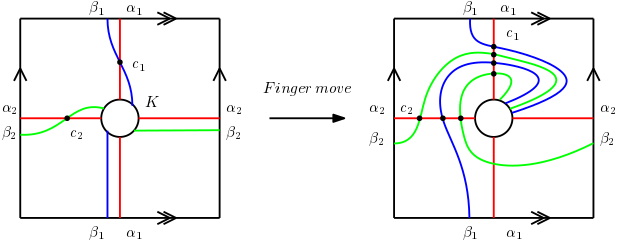}
\centering
\caption{$\Sigma$ is shown on the left before the finger move, and on the right after. New intersection points arise between the $\alpha$ and $\beta$ curves.}\label{fig:newintersection}
\end{figure}

If $\phi$ is not right-veering, then from \cite{HKM09} there exists a non-separating arc $a_{1}$ such that $\phi (a_{1})$ is sent to the left of $a_{1}$. Since $a_1$ is non-separating, the arc $a_{1}$ can be completed to a basis $\{a_{1},...a_{2g}\}$ for $\Sigma$. We form the Heegaard diagram according to the above procedure for doubling the open book, starting with this basis $\{a_{1},...a_{2g}\}$.  

\begin{proposition}[{\cite{BVV}}]\label{prop:c} 
Suppose the monodromy $\phi$ of a fibered knot $K\subset Y$ is not right-veering. Let $\cc =  \{c_1, \dots, c_{2g} \} \in \widehat{\CFK}(\mathcal{H})$ be the contact generator for $-K \subset -Y$. Then $A(\cc) = -g$, where $g$ is the genus of $K$. Furthermore, there exists a generator $\mathbf{y} \in \widehat{\textit{CF}}(\mathcal{H})$ with Alexander grading $A(\mathbf{y})=1-g$ and $\widehat{\partial}(\mathbf{y}) = \cc$. 
\end{proposition}

\begin{proof} The existence of the generator $\mathbf{y} = \{ d_1, c_2, \dots, c_{2g}\}$ with $\widehat{\partial}(\mathbf{y}) = \cc$ and Alexander grading $A(\mathbf{y})=1-g$ follows from the proof of Theorem 1.1 of \cite{BVV}: see Figure \ref{rvfigure}. Indeed, if $\phi$ is not right-veering, then $\alpha_{1}$ and $\beta_{1}$ form a bigon with corners at $c_{1}$, $d_{1}\subset \alpha_{1}\cap\beta_{1}$. Notice here we apply $\phi$ on $-\Sigma$, so $\phi (a_{1})$ is to the right of $a_{1}$. Every other holomorphic disk with corners containing $\{c_{2},...c_{2g}\}$ must contain $w$, so $\widehat{\partial}({\mathbf{y}})=\cc$. See \cite{BVV} for details.
\end{proof}

\begin{proposition} 
\label{prop:contactupsilon}
Suppose $\xx$ is a generator of $\widehat{\textit{CF}}(\mathcal{H})$ for $-K \subset -Y$ with $A(\xx) = -g$. Let $\xx \neq \cc$, where $\cc = \{ c_1, \dots, c_{2g}\} $ is the contact generator.
Then there exists a filtered chain homotopy equivalence 
\[\CFK^\infty(K) \simeq J \oplus A\]
where $A$ is acyclic, that is, $H_*(A) = 0$, and the image of $\xx$ lies in the acyclic summand $A$.
In particular, the generator $\xx$ cannot contribute to $\Upsilon_{-Y, -K, \s}(t)$.
\end{proposition}
\begin{proof}
We abbreviate $\CFK^\infty(-Y, -K) = C$ and let $C^\vertical = C\{ i = 0 \}$ and let $\partial = \partial^\vertical$ be the induced differential. 
As in the proof of Lemma 2.1 of \cite{Hom}, we let
\[B_0 = \{ \xx_i \in C^\vertical \mid \partial \xx_i \neq 0, \text{ and } A(\partial \xx_i) = A(\xx_i) \}.\]
Note that the contact generator $\cc \notin B_0$, as $\cc$ is a cycle in $C^\vertical \cong \widehat{\textit{CF}}(-Y)$ and $\widehat{\partial} \cc = 0$. 
As in the proof of Lemma 2.1 of \cite{Hom}, we have a filtered change of basis for $\CFK^{-}(-Y, -K)$ to a complex such that the acyclic summand $B_0 \cup \partial B_0$ may be discarded. This filtered change of basis for $\CFK^\infty(-Y, -K) \cong \CFK^{-}(-Y, -K) \otimes \F[U, U^{-1}]$ gives a filtered chain homotopy equivalence to a complex $J \oplus A$, where $A \supset B_0 \cup \partial B_0$ is acyclic, i.e. $H_*(A) = 0$.

By \cite{BVV}, $H_*(\mathcal{F}_{-g}) = \F\langle \cc \rangle$. This implies any $\xx \neq \cc$ with $A(\xx) = -g$ is an element of $B_0 \cup \partial B_0$. Thus, such $\xx$ will be discarded in the filtered change of basis. Furthermore, such $\xx$ cannot contribute to $\Upsilon_{-Y,-K}(t)$, as the $\Upsilon_{-Y,-K}(t)$ invariant only depends on the knot Floer complex up to stable equivalence \cite{AlfieriCeloriaStipsicz}. 
\end{proof}

\section{Main Theorem}

In this section we prove the main theorem, which we restate here for convenience:

\newtheorem*{mainthm}{Theorem~\ref{thm:main}}
\begin{mainthm}
If $\Upsilon'_{K,\mathfrak{s}}(t)=-g$ for some $t\in [0,1)$, where $g$ is the genus of the fibered surface $\Sigma$, then the monodromy $\phi_K:\Sigma\rightarrow\Sigma$ is right-veering.
\end{mainthm}

In the statement of the theorem, we are assuming that $t$ is a non-singular point of $\Upsilon_K(t)$. 

\begin{proof}

Suppose the monodromy $\phi_K \colon \Sigma\rightarrow\Sigma$ associated to $K \subset Y$ is not right-veering. We will prove that $\Upsilon_{Y, K, {\mathfrak{s}}}'(t)\neq -g$ for $t\in [0,1)$. In fact, we will show that $\Upsilon_{-Y, -K,\overline{\mathfrak{s}}}'(t)\neq g$ for $t\in [0,1)$. The statement then follows from the behavior of $\Upsilon$ under mirroring \ref{basic}: $\Upsilon_{Y, K,\mathfrak{s}}(t)=-\Upsilon_{-Y, -K,\overline{\mathfrak{s}}}(t)$. When $Y = S^3$, $-K \subset -Y$ represents the mirror of $K$. 

Now we consider the complex $CFK^{\infty}(-Y, -K,\overline{\mathfrak{s}})$ associated to the Heegard diagram compatible with the non-right-veering open book $(\Sigma,\phi_K)$ as described in the previous section. 

Let  $\mathbf{c}$ be the contact generator from Proposition \ref{prop:c} with $A(\mathbf{c})=-g$. Throughout, $A(\mathbf{x})$ refers to the Alexander grading of a generator $[\mathbf{x}, 0, A(\mathbf{x})]$ of the knot Floer chain complex $\CFK^\infty(K)$.

Let $\eta\in CFK^{\infty}(-Y, -K,\mathfrak{s})$ be a cycle satisfying:
\begin{itemize}
\item $0 \neq [\eta]\in \textit{HFK}^{\infty}(-Y, -K,\overline{\mathfrak{s}})$ has Maslov grading $d(-Y,\overline{\mathfrak{s}})$. 
\item $\eta = U^{-m}\mathbf{c}+\eta'$, where the integer $m$ denotes the algebraic filtration level of the generator $U^{-m}\cc=[\cc, m, -g+m]$.
\end{itemize}
Note that since the $U$-action decreases Maslov grading by 2, the integer $m$ is uniquely determined. If such a cycle $\eta$ doesn't exist, then $U^{-m}\cc$ never realizes $\nu_{t}(-Y, -K,\overline{\mathfrak{s}})$ for any integer $m$. As a result, $\Upsilon_{-K,\overline{\mathfrak{s}}}'(t)\neq g$ for any $t\in [0,2]$. Thus, we may now assume such a cycle $\eta$ exists.

If $\phi_{K}:\Sigma\rightarrow\Sigma$ is not right-veering, then by Proposition \ref{prop:c} there is a generator $[\yy,0,A(\yy)]$ of $CFK^{\infty}(-Y, -K,\overline{\mathfrak{s}})$ such that 

\begin{itemize}
\item $\widehat{\partial}(\mathbf{y})=\mathbf{c}$,
\item A$(\mathbf{y})=1-g$.
\end{itemize}

Suppose $\partial^{\infty}\mathbf{y}=\cc + \rho$. If $\rho=0$, then $U^{-m}\cc$ is a boundary in $CFK^{\infty}(-Y,-K,\s)$ and does not contribute to $\Upsilon_{-Y,-K,\overline{\mathfrak{s}}}(t)$. 
Thus, we can assume $\rho\neq 0$. Since the differential respects the $i$ and $j$ filtration, $\rho\in C\{i\leq -1, j\leq 1-g\}$. Hence, $U^{-m}\rho\in C\{i\leq m-1,j\leq 1-g+m\}$, and  \[F_{t}(U^{-m}\rho)\leq (1-\frac{t}{2})(m-1)+\frac{t}{2}(1-g+m).\]
Therefore,
\[F_{t}(U^{-m}\rho)-f_{t}(U^{-m}\cc) \leq 
(1-\frac{t}{2})(m-1)+\frac{t}{2}(1-g+m)-(-\frac{t}{2}g+m) 
=t-1.
\]
It follows that when $t < 1$, we have 
\[F_t(U^{-m} \rho) < f_t(U^{-m} \cc).\]

Since $\partial^{\infty}$ is a differential, we have 
\begin{align}\label{cycle1}
\partial^{\infty}(\partial^{\infty}U^{-m}\mathbf{y})=\partial^{\infty}U^{-m}\mathbf{c}+\partial^{\infty}U^{-m}\rho=0.
\end{align}
Recall that $\eta=U^{-m}\cc+\eta'$ is also a cycle, hence 
\begin{align}\label{cycle2}
\partial^{\infty}\eta=\partial^{\infty}U^{-m}\mathbf{c}+\partial^{\infty}\eta'=0
\end{align}
as well. By adding equations (\ref{cycle1}) and (\ref{cycle2}) we have $$\partial^{\infty}(U^{-m}\rho +\eta')=0.$$
Therefore, $U^{-m}\rho+\eta' $ is also a cycle in $CFK^{\infty}(-Y, -K,\overline{\mathfrak{s}})$, denoted by $\delta$. Furthermore, $\delta = \eta +\partial^{\infty}U^{-m}\yy$; thus $[\delta]=[\eta]\neq 0$ has Maslov grading $d(-Y, \overline{\mathfrak{s}})$.

Let $t<1$. Recall $F_t(U^{-m} \rho) < f_t(U^{-m} \cc)$. Since $F_{t}(\eta)=\max(F_{t}(\eta'),f_{t}(U^{-m}\cc))$ and $F_{t}(\delta)= \max(F_{t}(\rho),F_{t}(\eta'))$, by a case-by-case analysis,
\[F_t(\eta) \geq F_t(\delta) \geq \nu_t(-Y, -K, \overline{\mathfrak{s}}).\] 
The cycle $\delta$ does not contain the contact generator $U^{-m}\cc$ as a summand. Thus, by Proposition \ref{prop:contactupsilon}, no summand of $\delta$ that realizes $\nu_t(-Y,-K, \overline{\mathfrak{s}})$ lies in Alexander grading $-g$. Thus, if $\delta$ realizes $\nu_t(-Y,-K, \overline{\mathfrak{s}})$, then the slope of $\Upsilon_{-Y,-K, \overline{\mathfrak{s}}}(t)$ cannot equal $g(K)$ for such $t$. Since such a cycle $\delta$ exists for any choice of cycle $\eta$, we see that the slope of $\Upsilon_{-Y,-K, \overline{\mathfrak{s}}}(t)$ cannot equal $g(K)$ when $t < 1$. 
\end{proof}

\section{Examples}

In this section we investigate the strength of Theorem \ref{thm:main}. Recall from the Introduction that results of Hedden in \cite{Hedden} and Ozsv\'{a}th, Stipsicz, and Szab\'{o} in \cite{OSS} combined imply that for a fibered knot $K$ in $S^{3}$, the following three conditions are equivalent: (1) $K$ is strongly quasi-positive, (2) $\Upsilon_{K}(t) = -\tau(K)t = -g(K)t$ for small $t$, and (3) the fibration is compatible with the tight contact structure on $S^{3}$ (which implies that the monodromy is right-veering). A natural question to ask is whether the slope of $\Upsilon$ is always minimal close to $t=0$, because if so, then for fibered knots in $S^{3}$, Theorem \ref{thm:main} would not contain any new information. Example \ref{ex:cable} shows that this is not the case: we demonstrate that there are infinitely many knots where the slope of $\Upsilon$ only achieves the maximal slope of $-g$ away from $t=0$. In particular, this shows that $\Upsilon$ can detect information about the monodromy of a fibered knot that is not detected by $\tau$. However, in Example \ref{ex:small} we show that for small fibered knots in $S^{3}$, $\Upsilon$ always achieves minimal slope close to $t=0$. Finally Example \ref{ex:converse} shows that the converse of Theorem \ref{thm:main} does not hold.

\begin{example}\label{ex:cable}
The $(2, 2n+1)$--cables $T(2,-3)$, for $n \geq 8$, are fibered with right--veering monodromy. Their fibrations support overtwisted contact structures on $S^{3}$.
\end{example}
\begin{proof}
We remark that fiberedness of a cable $K_{p,q}$ of a fibered knot $K$ is well-known: it follows, for instance,
from \cite{Stallings} or from explicitly building the fibration of $Y\setminus K_{p,q}$ from the fibration of the companion $K$ and the fibration of the pattern $(p,q)$-torus knot. It is also well-known that the degree of the Alexander polynomial of a fibered knot detects the genus, and that the Alexander polynomial of a cable knot is determined by
\[\Delta_{K_{p,q}}(t)  = \Delta_K(t^p) \cdot \Delta_{T_{p,q}}(t).\]
The genus of the torus knot $T_{2,2n+1}$ is $n$, and the genus of $T_{2,-3}$ is 1. Thus, the genus of ${T_{2,-3;2,2n+1}}$ is $n+2$.

The above two facts (of fiberedness and the genus) can also be seen using knot Floer techniques:
Hedden in \cite{hedden2005knot} computed $\widehat{HFK}$ for these families (see in particular Proposition 4.1 in \cite{hedden2005knot}). His calculations give us that that the top Alexander grading where $\widehat{HFK}$ is nonzero is $n+2$ and that furthermore $\widehat{HFK}$ is rank one in Alexander grading $n+2$. The top Alexander grading where $\widehat{HFK}$ is nonzero is the genus of the knot (Theorem 1.2 from Ozsv\'{a}th and Szab\'{o} in \cite{OS04b}) and Ni proved in \cite{ni2007knot} that if $\widehat{HFK}$ is rank one in Alexander grading the genus of the knot, then it is fibered. All together this tells us that these knots are fibered and have genus $2+n$.

The Upsilon invariant for this family of cable knots was computed by Chen in \cite{Chen} as follows:

\[
  \Upsilon_{(T_{2,-3})_{2,2n+1}}(t) = 
  \begin{cases}
                                   -(n-1)t & \text{if $t \in [0,\frac{2}{3}]$} \\
                                   2-(n+2)t & \text{if $t \in [\frac{2}{3},1]$} \\
  \end{cases}
\]

Notice that the slope of $\Upsilon$ is not minimal at $t=0$, which implies by Hedden \cite{Hedden} that the fibrations of these knots must support overtwisted contact structures. However, the slope of $\Upsilon$ is minimal (that is, the negative of the genus) on $[\frac{2}{3},1]$, and so Theorem \ref{thm:main} shows that the monodromies of these knots must be right-veering. 

Alternatively, one could prove the right-veering statement by using the fractional Dehn twist coefficients of the knots, which are positive for a positive cable by \cite{Kazez-Roberts}. 
\end{proof}

\begin{example}\label{ex:small}
For fibered knots $K\subset S^{3}$ with less than 10 crossings, $\Upsilon_{K}'(t)=-g(K)$ for some $t\in [0,1)$ if and only if $K$ supports the unique tight contact structure on $S^{3}$.
\end{example}

\begin{proof} For any knot in $S^{3}$, Ozsv\'{a}th, Stipsicz, and Szab\'{o} \cite{OSS} prove that $\Upsilon_{K}(t)=-\tau (K)t$ for small $t$. Moreover, if the fibered knot supports the tight contact structure on $S^{3}$, then $\tau (K)=g(K)$ by Hedden \cite{Hedden}.

For the other direction, we show the contrapositive. Again by Hedden, a fibered knot $K$ supports the tight contact structure in $S^{3}$ if and only if it is strongly quasi-positive. The monodromies of fibered knots under ten crossings that are not strongly quasi-positive are available in \cite{KnotInfo}. By brute force one can see that none of them are right-veering by finding arcs that are sent to the left. Therefore, by Theorem \ref{thm:main}, $\Upsilon'_{K}(t)>-g(K)$. 
\end{proof}

\begin{example} \label{ex:converse}
The converse of Theorem \ref{thm:main} does not hold even for fibered knots in $S^{3}$. 
\end{example}

\begin{proof} Consider the knot $K=8_{20}$, which is a slice and fibered knot. Recall that cabling preserves fiberedness (see e.g. \cite{Hedden-remarkscabling}). Also recall that cabling induces a map on the concordance group (see e.g. \cite{Hedden-PC}).  Thus, the $(p,1)-$cable $K_{p,1}$ of $K$ is also slice and fibered. Indeed, $8_{20}$ is the pretzel $P(3,-3,2)$. One can construct a slice disk by adding two 1-handles and three 2-handles in $B^{4}$. The slice disk of the $(p,1)-$cable can be obtained by stacking $p$ copies of the disks constructed above and connecting them with half-twisted bands. Therefore, $\Upsilon_{K_{p,1}}(t)=0$. On the other hand, Kazez and Roberts \cite{Kazez-Roberts} show that the fractional Dehn twist coefficient of a fibered knot obtained by $(p,1)$--cabling is $\frac{1}{p}>0$. Hence the monodromy of $K_{p,1}$, for $p>0$, is right-veering.
\end{proof}

\section{Applications to Concordance and the Slice-Ribbon Conjecture}\label{sec:applications}

As pointed out in the Introduction, fibered knots with right-veering monodromy can be concordant to fibered knots with non-right-veering monodromy. For instance, let $K$ denote the $(p,1)$-cable of any fibered slice knot. Then $K$ is right-veering since its fractional Dehn twist coefficient $\frac{1}{p}$ is strictly greater than zero. Let $J$ be the mirror of $K$, which has negative fractional Dehn twist coefficient and therefore has non-right-veering monodromy. The knots $K$ and $J$ have the same genus and are both slice since cabling is an operator on the concordance group (see e.g. \cite{Hedden-PC}). Thus, the knot $K$ with right-veering monodromy is concordant to the knot $J$ with non-right-veering monodromy. 

Theorem \ref{thm:main} immediately gives an obstruction to concordance among some fibered knots, which we restate from the Introduction:

\newtheorem*{propconcordanceobst}{Proposition~\ref{prop:concordanceobst}}
\begin{propconcordanceobst} Let $K$ in $S^{3}$ be a fibered knot of genus $g$ such that $\Upsilon'_{K}(t_{0}) = -g(K)$ for some $t_{0} \in [0,1)$. Then $K$ cannot be concordant to any fibered knot in $S^{3}$ of the same genus whose monodromy is not right-veering. 
\end{propconcordanceobst} 
\begin{proof} If $K$ is concordant to a fibered knot $J$ of the same genus, then $\Upsilon_{J}(t) = \Upsilon_{K}(t)$ for all $t \in [0,2]$. Thus $\Upsilon'_{K}(t_{0}) = \Upsilon'_{J}(t_{0}) = -g(J)$, implying that the monodromy of $J$ is right-veering by Theorem \ref{thm:main}. \end{proof}

Proposition \ref{prop:concordanceobst} shows that for a special class of fibered knots with right-veering monodromies, namely those for which $\Upsilon_{K}'(t_{0}) = -g(K)$ for some $t_{0} \in [0,1)$, any concordance to a knot of the same genus will preserve the right-veering property. Understanding when and how monodromies of fibered knots can restrict concordance is an interesting question, and as far as we are aware, relatively unexplored up to now. We think this research direction deserves further exploration. 

Baker conjectures that tight, fibered knots are unique in their concordance class. 
\begin{conjecture}[\cite{baker}]
Let $K_{0}$ and $K_{1}$ be fibered knots in $S^{3}$ supporting the tight contact structure. If $K_{0}$ and $K_{1}$ are concordant, then $K_{0}=K_{1}$.
\end{conjecture}

Since tight, fibered knots satisfy the condition $\Upsilon'_{K}(t) = -g(K)$ for $t$ close to $0$, we are inspired by Baker to ask the following question from the Introduction:
\newtheorem*{quesliceribbon}{Question~\ref{que:sliceribbon}}
\begin{quesliceribbon} Suppose $K_{0}$ and $K_{1}$ are fibered knots in $S^{3}$ satisfying that for each $i \in \{0, 1\}$, $\Upsilon'_{K_{i}}(t_{i}) = -g(K_{i})$ for some $t_{i} \in [0,1)$. If $K_{0}$ and $K_{1}$ are concordant, is $K_{0} = K_{1}$? 
\end{quesliceribbon}

Supporting a positive answer to Question~\ref{que:sliceribbon} are Lemma~\ref{lem:genera} and Theorem~\ref{thm:ribbonequal}. 
\begin{lemma} \label{lem:genera}
Suppose $K_{0}$ and $K_{1}$ are fibered knots in $S^{3}$ satisfying that for each $i \in \{0, 1\}$, $\Upsilon'_{K_{i}}(t_{i}) = -g(K_{i})$ for some $t_{i} \in [0,1)$. If $K_{0}$ and $K_{1}$ are concordant, then $g(K_0) = g(K_1)$. 
\end{lemma}
\begin{proof}
Suppose without loss of generality that $-g(K_1) < -g(K_0)$.  
For all $t \in [0,2]$,
\[-g(K_0) \leq \Upsilon'_{K_0}(t) \leq g(K_0). \]
If $K_0$ and $K_1$ are concordant, then $\Upsilon'_{K_0}(t) = \Upsilon'_{K_1}(t)$ for all $t$.  In particular, \[ \Upsilon'_{K_0}(t_1) =  \Upsilon'_{K_1}(t_1) = -g(K_1) < -g(K_0),\] a contradiction.
\end{proof}

A positive answer to Question \ref{que:sliceribbon} would imply Baker's conjecture. We show in Theorem~\ref{thm:ribbonequal} that the answer to Question \ref{que:sliceribbon} is yes if concordance is replaced with $K_{0} \# -K_{1}$ being ribbon, and thus we observe:
\newtheorem*{corsliceribbon}{Corollary~\ref{cor:sliceribbon}}
\begin{corsliceribbon} If the answer to Question \ref{que:sliceribbon} is negative, then the Slice-Ribbon Conjecture is false.
\end{corsliceribbon}
We remark that it is possible that Baker's conjecture is true, but the answer to Question \ref{que:sliceribbon} is negative. Thus this question widens the avenue that Baker opened to disprove the Slice-Ribbon Conjecture. We closely follow \cite{baker} below.

\begin{definition}
Let $K_{0}$, $K_{1}$ be knots in homology 3-spheres $Y_{0}$ and $Y_{1}$. $(Y_{1},K_{1})$ is homotopy ribbon concordant to $(Y_{0},K_{0})$, denoted by $K_{1}\geq K_{0}$, if there exists some 4-dimensional manifold $X$ that is a homology $S^{3}\times [0,1]$ containing an embedded smooth annulus $A$ such that $(\partial X, A\cap\partial X)=(Y_{1},K_{1})\sqcup -(Y_{0},K_{0})$ with surjective $\pi_{1}(Y_{1}-K_{1})\rightarrow\pi_{1}(X-A)$ and injective $\pi_{1}(Y_{0}-K_{0})\rightarrow\pi_{1}(X-A)$. 
\end{definition}
Homotopy ribbon concordance generalizes the notion of ribbon concordance \cite[Lemma 3.1]{Gordon}. Following the terminology of \cite{HomLevineLidman}, we say that two knots $K_0 \subset Y_0$ and $K_1 \subset Y_1$ are homology concordant if there is a smoothly embedded annulus in some homology cobordism with boundary $(Y_0,K_{1})\sqcup -(Y_1,K_{0})$. For knots in $S^{3}$, if $K_{1}\geq K_{0}$, then $K_{1}$ and $K_{0}$ are homology concordant. 

\begin{theorem}[{\cite[Theorem 1.7]{HomLevineLidman}}]
If the knots $K_{0} \subset Y_0$ and $K_{1}\subset Y_1$ are homology concordant, then $\Upsilon_{K_{0}, Y_0}(t)=\Upsilon_{K_{1}, Y_1}(t)$.
\end{theorem} 

The following lemma generalize Baker's result that tight fibered knots are minimal with respect to homotopy ribbon concordance among fibered knots in $S^{3}$ \cite[Lemma 2]{baker}. 
\begin{lemma} \label{lem:ribbon}
Let $K$ be a fibered knot in $S^{3}$. If $\Upsilon_{K}'(t)= - g(K)$ for some $t \in [0,2]$, then $K$ is minimal with respect to homotopy ribbon concordance among fibered knots in $S^{3}$.
\end{lemma}

\begin{proof}
If a fibered knot $K$ in $S^3$ is homotopically ribbon concordant to a fibered knot $J$ in $S^3$, then Gordon \cite[Lemma 3.4]{Gordon} showed that $d(K) \geq d(J)$, where $d$ denotes the degree of the Alexander polynomial, and because $K$ and $J$ are fibered knots, we have $g(K) \geq g(J)$. Since $K$ is homotopically ribbon concordant to $J$, the knots $K$ and $J$ are in particular homology concordant, and thus $\Upsilon_K(t) = \Upsilon_J(t)$ by \cite[Theorem 1.7]{HomLevineLidman}. By the hypothesis, there exists some $t \in [0,2]$ such that $- g(K) = \Upsilon_K'(t) = \Upsilon_J'(t) \geq - g(J) $, which implies that $g(K) = g(J)$.  Then $d(K) = d(J)$, and by \cite[Lemma 3.4]{Gordon},  we have $K = J$. 
\end{proof}

Similarly, we can prove that mirrors of knots satisfying this $\Upsilon$ condition are minimal with respect to homotopy ribbon concordance among fibered knots. 

\begin{lemma} \label{lem:mirrorminimal}
Let $K$ be a fibered knot in $S^{3}$. If $\Upsilon_{K}'(t)= - g(K)$ for some $t \in [0,2]$, then the mirror $-K$ of $K$ is minimal with respect to homotopy ribbon concordance among fibered knots in $S^{3}$.
\end{lemma}
\begin{proof}
If the fibered knot $-K$ in $S^3$ is homotopically ribbon concordant to a fibered knot $J$ in $S^3$, then Gordon \cite[Lemma 3.4]{Gordon} showed that $d(-K) \geq d(J)$, where $d$ denotes the degree of the Alexander polynomial, and because $-K$ and $J$ are fibered knots, we have $g(-K) \geq g(J)$. Since $-K$ is homotopically ribbon concordant to $J$, the knots $-K$ and $J$ are in particular homology concordant, and thus $\Upsilon_{-K}(t) = \Upsilon_J(t)$ by \cite[Theorem 1.7]{HomLevineLidman}. By the hypothesis, there exists some $t \in [0,2]$ such that \[-g(-K) = -g(K) = \Upsilon_K'(t) = -\Upsilon_{-K}'(t) = - \Upsilon_J'(t) \geq - g(J),\] which implies that $g(-K) = g(J)$.  Then $d(-K) = d(J)$, and by \cite[Lemma 3.4]{Gordon},  we have $-K = J$. 
\end{proof}

\begin{theorem}\label{thm:ribbonequal}
Let $K_{0}$ and $K_{1}$ be fibered knots in $S^{3}$ such that $\Upsilon'_{K_{i}}(t_{i}) = -g(K_{i})$ for some $t_{i} \in [0,1]$, for each $i \in \{0, 1\}$. If $K_{0}\# -K_{1}$ is ribbon, then $K_{0}=K_{1}$.
\end{theorem}
\begin{proof}
The proof of \cite[Theorem 3]{baker} carries through verbatim, relying on our Lemma \ref{lem:ribbon} and Lemma \ref{lem:mirrorminimal} instead of \cite[Lemma 2]{baker}. 
\end{proof}


\bibliographystyle{alpha}
\bibliography{bib}

\end{document}